\documentclass[11pt]{article}
\usepackage{amssymb,amsmath,amsthm,amsfonts,verbatim}
\usepackage[all,cmtip]{xy}
\usepackage{hyperref}
\usepackage[top=1.2in,bottom=1.2in,left=1.21in,right=1.21in]{geometry}
\usepackage{todonotes}
\usepackage{qtree}
\usepackage{textcomp}

\newtheorem{thm}{Theorem}[section]
\newtheorem{prop}[thm]{Proposition}
\newtheorem{cor}[thm]{Corollary}
\newtheorem{lem}[thm]{Lemma}

\newtheorem{thmx}{Theorem}

\newtheorem{definition}[thm]{Definition}

\theoremstyle{definition}

\newtheorem{example}[thm]{Example}
\newtheorem{fact}[thm]{Fact}
\newtheorem{note}[thm]{Note}

\newtheorem{notation}[thm]{Notation}
\newtheorem{question}[thm]{Question}
\newtheorem{claim}{Claim}

\numberwithin{equation}{section}

\theoremstyle{remark}

\newtheorem{remark}[thm]{Remark}

\usepackage{appendix}
\usepackage{extarrows}
\newcommand{\F}{\mathbb{F}}  
\newcommand{\N}{\mathbb{N}}
\newcommand{\Q}{\mathbb{Q}}
\newcommand{\Z}{\mathbb{Z}}  
\newcommand{\A}{\mathbb{A}}  

\newcommand{\mor}[1]{\ensuremath{\overset{#1}{\longrightarrow}}} 

\DeclareMathOperator{\Ho}{H} 
\DeclareMathOperator{\Fr}{Frob} 
\DeclareMathOperator{\Tr}{Tr} 
\DeclareMathOperator{\Stab}{Stab} 
\DeclareMathOperator{\Aut}{Aut} 
\DeclareMathOperator{\Irr}{Irr} 
\DeclareMathOperator{\Poly}{Poly} 

\newcommand{\catname}[1]{{\normalfont\operatorname{#1}}}
\newcommand{\FI}{\catname{FI}}

\begin{document}


\title{A trace formula for the distribution of rational $G$-orbits in ramified covers, adapted to representation stability}


\author{Nir Gadish}






\maketitle

\begin{abstract}
A standard observation in algebraic geometry and number theory is that a ramified cover of an algebraic variety $\widetilde{X}\rightarrow X$ over a finite field $\F_q$ furnishes the rational points $x\in X(\F_q)$ with additional arithmetic structure: the Frobenius action on the fiber over $x$. For example, in the case of the Vieta cover of polynomials over $\F_q$ this structure describes a polynomial's irreducible decomposition type.

Furthermore, the distribution of these Frobenius actions is encoded in the cohomology of $\widetilde{X}$ via the Grothendieck-Lefschetz trace formula. This note presents a version of the trace formula that is suited for studying the distribution in the context of representation stability: for certain sequences of varieties $(\widetilde{X}_n)$ the cohomology, and therefore the distribution of the Frobenius actions, stabilizes in a precise sense.

We conclude by fully working out the example of the Vieta cover of the variety of polynomials. The calculation includes the distribution of cycle decompositions on cosets of Young subgroups of the symmetric group, which might be of independent interest.
\end{abstract}





\section{Introduction}
The increasingly important notion of representation stability introduced by Church-Farb \cite{CF} identifies sequences of spaces with group actions $(G_n\curvearrowright X_n)_{n\in \N}$ whose cohomology groups exhibit a kind of stabilization as representations for $n\rightarrow \infty$. One then hopes to translate the observed cohomological stabilization into arithmetic results via the bridge provided by the Grothendieck-Lefschetz trace formula. This was realized e.g. by Church-Ellenberg-Farb \cite{CEF-pointcounts} in the case of statistics of square-free polynomials and maximal tori in $\operatorname{Gl}_n$ over finite fields. One difficulty that this program faces is the possible presence of nontrivial stabilizers of the group actions, and their effect on the trace formula. This article offers a treatment of actions with stabilizers and the adaptation of the trace formula to representation stability applications. The formula, presented in Theorem \ref{thm:intro_main} below, is proved using standard methods and will not be considered new by algebraic-geometers\footnote{For example, the same ideas and definitions appear in Grothendieck's \cite{Groth} and in Serre's \cite{Serre}.}. Rather, it is presented as a `ready for use' tool to be applied in the context of representation stability.

Using the approach presented here, we extend the project initiated in \cite{CEF-pointcounts} to include the statistics of polynomials with possible root multiplicities (see details in \S\ref{subsec:intro-polynomials}). Let us remark that much of the work that goes into polynomial statistics often passes through calculations on square-free polynomials and ignores the rest (the latter being relatively uncommon), see e.g. \cite{ABR}[Section 4]. The calculations below suggest a way to handle more general polynomials: we introduce an algebra of \emph{division symbols} on the space of polynomials, and show that these give rise to functions that serve as a direct link between the statistics of polynomials and those of symmetric groups (see subsection \S\ref{subsec:intro-polynomials}).

\subsection{Distribution of rational orbits}\label{subsec:intro-trace}
Let $\widetilde{X}$ be an algebraic variety over a finite field $\F_q$, endowed with an action of a finite group $G$. Then the variety of orbits $X = \widetilde{X}/G$ acquires arithmetic information from $\widetilde{X}$: a rational point $x\in X(\F_q)$ corresponds to a $G$-orbit of $\widetilde{X}(\overline{\F}_q)$ that is stable under the Frobenius automorphism $\Fr_q$. Thus for every $x\in X(\F_q)$ the Frobenius determines a $G$-equivariant permutation $\sigma_x$ on a transitive $G$-set, and this additional information distinguishes rational points in $X$ in a subtle way. For example, let $X$ be the space of monic degree $d$ polynomials. Ordering the roots of a polynomial gives that $X$ is the quotient $\A^d/S_d$ with $S_d$, the symmetric group on $d$ letters, acting by permuting the entries of $\A^d$. Then for every polynomial $f(t)\in \F_q[t]$ the permutation $\sigma_f$ encodes precisely the decomposition type of $f$ into irreducible factors over $\F_q$.

On the other hand, following the philosophy of the Weil conjectures, it is known that the action of $\Fr_q$ on the \'{e}tale cohomology groups $\Ho^*_{\acute{e}t}(\widetilde{X}_{/\overline{\F}_q};\Q_\ell)$ encodes arithmetic information. At the same time, an action $G\curvearrowright \widetilde{X}$ induces a $G$-representation on $\Ho^*(\widetilde{X})$, and it is natural to ask: What arithmetic information is encoded by the joint action of $\Fr_q$ and $G$ on $\Ho^*(\widetilde{X})$?

One answer, given below, is that the information encoded in $\Ho^*(\widetilde{X})$ is in some sense the distribution of the permutations $\sigma_x$ attached to rational points $x\in X(\F_q)$. However, it is not initially clear what one should mean by a ``distribution" of permutations $\sigma_x$ on abstract $G$-orbits. The Tanakian point of view tells us instead to examine how $\sigma_x$ acts on $G$-representations, or equivalently: how it evaluates on $G$-characters.

We therefore detect the distribution of the permutations $\sigma_x$ by evaluating them on class functions of $G$ as follows. Let $\chi:G\mor{} \mathbb{C}$ be a class function. If the quotient map $p:\widetilde{X}\mor{}X$ is unramified at $x\in X(\F_q)$ (i.e. the stabilizer of a lift $\widetilde{x}\in p^{-1}(x)$ is trivial), then $\sigma_x$ determines an element $g_x\in G$, unique up to conjugacy, by $\sigma_x(\widetilde{x})=g_{x}.\widetilde{x}$ for a chosen lift $\widetilde{x}\in p^{-1}(x)$. Changing the lift $\widetilde{x}$ only amounts to conjugating $g_x$, so it is possible to unambiguously define $\chi(\sigma_x) := \chi(g_x)$.

However, when $p$ is ramified at $x$, the permutation $\sigma_x$ no longer determines a conjugacy class: if $H_{\widetilde{x}}\subseteq G$ is the stabilizer of a lift $\widetilde{x}\in p^{-1}(x)$, then the condition $\sigma_x(\widetilde{x}) = g_{x}.\widetilde{x}$ only determines a coset $g_{x}H_{\widetilde{x}}$. The best one can do in this situation is to average:
\begin{definition}[\textbf{Evaluating $\sigma_x$ on class functions}] \label{def:intro_evaluate_sigmax}
Let $\chi:G\mor{} \mathbb{C}$ be a class function. For $x\in X(\F_q)$ and a lift $\widetilde{x}\in p^{-1}(x)$ with stabilizer $H_{\widetilde{x}}$ define
$$
\chi(\sigma_x) := \frac{1}{|H_{\widetilde{x}}|}\sum_{h\in H_{\widetilde{x}}}\chi(g_{x}h)
$$
where $g_{x}\in G$ is any element satisfying $\sigma_x(\widetilde{x})=g_{x}.\widetilde{x}$.
\end{definition}
Theorem \ref{thm:intro_main} below relates the sum $\displaystyle{ \sum_{x\in X(\F_q)} \chi(\sigma_x)}$ to the representation $\Ho^*(\widetilde{X})$. As $\chi$ ranges over all class functions, these sums in some sense capture the distribution of $\sigma_x$.

\begin{thmx}[\textbf{Frobenius distribution trace formula}]\label{thm:intro_main}
Let $G$ be a finite group, acting on an algebraic variety $\widetilde{X}$ over the finite field $\F_q$, and let $X=\widetilde{X}/G$ as above. Fix a prime $\ell\gg 1$ coprime to $q|G|$ and let $\Ho^i_c(\widetilde{X})$ denote the compactly supported $\ell$-adic cohomology $\Ho_{c,\acute{e}t}^i(\widetilde{X}_{\overline{\F}_q};\Q_\ell)$. Decompose $\Ho^i_c(\widetilde{X})\otimes \overline{\Q}_\ell$ into generalized eigenspaces\footnote{Here ``generalized eigenspaces" means all possibly nontrivial Jordan blocks. This is to avoid questions of possible non-semisimplicity of the Galois action, which are not of interest to us in the current context.} of the $\Fr_q$ action:
$$
\Ho^i_c(\widetilde{X})\otimes \overline{\Q}_\ell = \bigoplus_{\lambda\in \overline{Q}_\ell} \Ho^i_c(\widetilde{X})_{\lambda}.
$$
Note that this sum includes only finitely many nonzero summands, and that each $\Ho^i_c(\widetilde{X})_{\lambda}$ is a $G$-subrepresentation. Then for every class function $\chi:G\mor{}\overline{\Q}_\ell$,
\begin{equation} \label{eq:intro_G-L-trace-formula_cc}
\sum_{x\in X(\F_q)} \chi(\sigma_x) = \sum_{\lambda\in \overline{\Q}_\ell} \lambda \sum_{i=0}^{\infty} (-1)^i \langle \Ho^i_c(\widetilde{X})^*_\lambda , \chi \rangle_G
\end{equation}
where the inner product $\langle V,\chi\rangle_G$ for a $G$-representation $V$ is the standard character inner product of $\chi$ with the character of $V$. Note again that this is really a finite sum.

When $\widetilde{X}$ is further taken to be smooth, Poincar\'{e} duality implies
\begin{equation} \label{eq:intro_G-L-trace-formula}
\sum_{x\in X(\F_q)} \chi(\sigma_x) = q^{\dim(\widetilde{X})}\sum_{\lambda} \lambda^{-1} \sum_{i=0}^{\infty} (-1)^i \langle \Ho^i(\widetilde{X})_\lambda , \chi \rangle_G.
\end{equation}
\end{thmx}

\subsection{Implications of representation stability}
Without going into the full detail of the \cite{CEF-FI} theory of representation stability, it provides examples of sequences of spaces $(\widetilde{X}_n)_{n\in \N}$ where the symmetric group $S_n$ acts on $\widetilde{X}_n$, and where the induced representations $H^i(\widetilde{X}_n)$ stabilize in the following sense: if $(\chi_n:S_n\rightarrow \Q)_{n\in \N}$ is a certain natural sequence of class functions (namely, given by a character polynomial, defined explicitly below) then the character inner products 
\begin{equation}\label{eq:intro_repstability}
\langle H^i(\widetilde{X}_n), \chi_n\rangle_{S_n}
\end{equation}
become independent of $n$ for $n\gg 1$. In the algebraic setting, the same stabilization occurs within every eigenspace of $\Fr_q$ (this observation follows immediately from the Noetherian property of the category $\FI$, see \cite{CEF-FI}). Denote the stable values of these inner products by
$$\langle H^i(\widetilde{X}_\infty)_\lambda, \chi_\infty\rangle.$$

This phenomenon was used in \cite{CEF-pointcounts}, along with the Grothendieck-Lefschetz trace formula in the unramified context, to demonstrate that the factorization statistics of degree $d$ square-free polynomials over $\F_q$ and maximal tori in $\operatorname{GL}_d(\F_q)$ tend to a limit as $d\rightarrow \infty$ (see \cite[Theorem 1 and 5.6 respectively]{CEF-pointcounts}.

A general type of result that one gets by combining representation stability and the trace formula of Theorem \ref{thm:intro_main} is the following.
\begin{cor}[\textbf{Limiting arithmetic statistics}]\label{cor:intro_repstability}
Let $(\widetilde{X}_n)_{n\in \N}$ be a sequence of smooth algebraic varieties over $\F_q$ where $S_n$ acts on $\widetilde{X}_n$, and denote the quotients by $X_n = \widetilde{X}_n/S_n$. Suppose that the cohomology $\Ho^i(\widetilde{X}_n)$ exhibits representation stability in the sense of \cite{CEF-FI}. Further suppose that $\Ho^*(\widetilde{X}_\bullet)$ is convergent in the sense of \cite[Definition 3.12]{CEF-pointcounts}. Then for every sequence of class functions $(\chi_n)_{n\in \N}$, given uniformly by a character polynomial, the following equality holds
\begin{equation}
\lim_{n\rightarrow \infty} q^{-\dim(\widetilde{X}_n)}\sum_{x\in X_n/S_n} \chi_n(\sigma_x) = \sum_{i\geq 0}\sum_{\lambda} \frac{(-1)^i}{\lambda} \langle \Ho^i(\widetilde{X}_\infty)_{\lambda},\chi_\infty\rangle.
\end{equation}
In particular, the limit on the left exists, and its value is given by the value of the convergent sum on the right.
\end{cor}
\begin{remark}[\textbf{Other sequences of groups}]
In \cite{Ga-FI} the author introduces \emph{categories of $\FI$-type}, to which the theory of representation stability can be extended. In particular, Corollary \ref{cor:intro_repstability} holds more generally for any diagram $\widetilde{X}_\bullet$ of varieties for which $\Ho^*(\widetilde{X}_\bullet)$ exhibits representation stability and has a subexponential bound on the growth of certain invariants.

For example, in \cite{Ga-arrangements} the author demonstrated that many collections of complements of linear subspace arrangements indeed give rise to cohomology groups which exhibit representation stability. Therefore, if such a collection satisfies an additional subexponential growth condition as above, then an analog of Corollary \ref{cor:intro_repstability} holds.
\end{remark}

\subsection{Example: Spaces of polynomials and Young cosets}\label{subsec:intro-polynomials}
We conclude in \S\ref{sec:polys} by fully working out the example of the space of polynomials $\Poly^d = A^d/S_d$ mentioned in the first paragraph. In this case, as introduced in the beginning of \ref{subsec:intro-trace}, the permutation $\sigma_f$ that the Frobenius induces on the roots of a polynomial $f\in \Poly^d(\F_q)$ records the irreducible decomposition of $f$. Thus Theorem \ref{thm:intro_main} is concerned with the fundamental question of factorization statistics of polynomials over $\F_q$.

On the one hand the variety $\widetilde{X}$ in this case is $\A^d$ and has very simple cohomology. On the other hand, the quotient map $p:\A^d\mor{} \Poly^d$ is highly ramified and the associated permutations $\sigma_f$ are very far from defining conjugacy classes of $S_d$. Thus \S\ref{sec:polys} deals mainly with the combinatorial challenge of evaluating $\chi(\sigma_f)$ when $f$ is a ramified point, i.e. computing averages of $\chi$ over cosets of Young subgroups of $S_d$. The same calculation is useful in many other contexts when the symmetric group acts by permutations.

Applying Theorem \ref{thm:intro_main} to this case produces the following apparent coincidence (which will become obvious once the two sides of Equation \ref{eq:equal-expectation} are evaluated):
\begin{cor}[\textbf{Equal expectations}]\label{cor:intro_equal_expectations}
Endow the two finite sets $S_d$ and $\Poly^d(\F_q)$ with uniform probability measures. Then every $S_d$-class function $\chi$ simultaneously defines a random variable on both spaces (for a point $f\in \Poly^d(\F_q)$ define $\chi(f):=\chi(\sigma_f)$ as in Definition \ref{def:intro_evaluate_sigmax}), and for every such $\chi$
\begin{equation}\label{eq:equal-expectation}
\mathbb{E}_{\Poly^d(\F_q)}\left[\chi\right] = \mathbb{E}_{S_d}\left[\chi\right].
\end{equation}
\end{cor}
To understand the left hand side of this equation, one has to interpret the value of $\chi(f)$ for every $f\in \Poly^d(\F_q)$. The explicit description of this value requires some notation, and the complete answer is given in Theorem \ref{thm:intro_evaluating_xmu} on the next page. For the necessary notation -- since $\chi(f)$ is related to the divisors $f$, it will be most convenient to describe its value using the following natural structure of \emph{division symbols}.

For every polynomial $g\in \F_q[t]$ define a function $\epsilon_g:\F_q[t]\mor{}\{0,1\}$ by
\begin{equation}\label{eq:epsilon_evaluation}
\epsilon_g(f) = \begin{cases}
1 \quad \text{if } g|f\\ 0 \quad \text{otherwise}.
\end{cases}
\end{equation}
Further define formal multiplication on the symbols $\epsilon_g$ by
\begin{equation}\label{eq:epsilon_multiplication}
\epsilon_g \epsilon_{g'} = \epsilon_{gg'}.
\end{equation}
Note that this multiplication does not commute with the evaluation on a polynomial $f$, and in fact every $\epsilon_g$ evaluates on $f$ nilpotently.

To evaluate $\chi(f)$ for every $\chi$, it suffices to consider a spanning set of class functions. The most convenient in this context are given by character polynomials (see \cite{CEF-pointcounts}) which are described in \S\ref{subsec:evaluating_chi} below. Briefly, for every $k$ let $X_k$ by the class function
$$
X_k(\sigma) = \# \text{ of $k$-cycles in }\sigma
$$
and for every multi-index $\mu=(\mu_1,\mu_2,\ldots)$ define
$$
\binom{X}{\mu} = \binom{X_1}{\mu_1}\binom{X_2}{\mu_2}\ldots.
$$
Explicitly, $\binom{X}{\mu}$ is the $S_d$-class function that counts the number of ways to choose $\mu_k$ many $k$-cycles in a permutation $\sigma\in S_d$. Instead of evaluating each $\binom{X}{\mu}$ separately, it is possible evaluate them all simultaneously be means of a generating function.
\begin{definition}
Introduce indeterminants $t_1,t_2,\ldots$, and define a generating function
$$
F(\textbf{t}) = \sum_{\mu = (\mu_1,\mu_2,\ldots)}\binom{X}{\mu}t_1^{\mu_1}t_2^{\mu_2}\ldots = (1+t_1)^{X_1}(1+t_2)^{X_2}\ldots
$$
\end{definition}
Evaluating $\binom{X}{\mu}(f)$ for every $\mu$ is the same as evaluating $F(\textbf{t})$ on $f$.
\begin{thm}[\textbf{Evaluation of $\binom{X}{\mu}$}]\label{thm:intro_evaluating_xmu}
When evaluating $f\in \Poly^d(\F_q)$ on class functions, there is an equality of generating functions
\begin{equation}\label{eq:intro_generating_function}
F(\textbf{t}) = \exp\left( \sum_{k=1}^\infty \sum_{p\in \Irr_{|k}}\deg(p)\epsilon_{p}^{\frac{k}{\deg(p)}} \, \frac{t_k}{k} \right)
\end{equation}
where $\Irr_{|k}$ is the set of irreducible polynomials over $\F_q$ whose degree divides $k$.

Unpacking Equation \ref{eq:intro_generating_function}, for every multi-index $\mu=(\mu_1,\mu_2,\ldots)$ there is an equality
\begin{equation} \label{eq:evaluation_on_xmu}
\binom{X}{\mu}(f) = \prod_{k=1}^\infty \frac{1}{k^{\mu_k}\mu_k!} \left( \sum_{p\in \Irr_{|k}} \deg(p)\epsilon_p^{\frac{k}{\deg(p)}} \right)^{\mu_k} (f)
\end{equation}
where the evaluation of the right-hand side proceeds by first expanding into monomial terms in the $\epsilon_g$ symbols using Equation \ref{eq:epsilon_multiplication}, then evaluating according to Equation \ref{eq:epsilon_evaluation}.
\end{thm}

With Theorem \ref{thm:intro_evaluating_xmu} it is possible explain the coincidence of the two expectations in Corollary \ref{cor:intro_equal_expectations}:
\begin{prop}[\textbf{Necklace relations}]\label{prop:intro_necklace}
The equalities $\displaystyle{\mathbb{E}_{\Poly^d(\F_q)}\left[\chi\right] = \mathbb{E}_{S_d}\left[\chi\right]
}$ for every $\chi$ are equivalent to the Necklace relations
$$
\sum_{d|k}d N_d = q^k
$$
where $N_d$ is the number of monic irreducible polynomials of degree $d$ over $\F_q$.
\end{prop}
Proposition \ref{prop:intro_necklace} shows that in fact Corollary \ref{cor:intro_equal_expectations} says nothing new about polynomials. However, the explicit evaluation of $\chi(f)$ in Theorem \ref{thm:intro_evaluating_xmu} contains much more information: one can impose any $S_d$-invariant restriction on the roots of polynomials (see examples in \S\ref{sec:polys}) and get an equality similar to Corollary \ref{cor:intro_equal_expectations} -- relating the factorization statistics of those polynomials that satisfy the restriction with various expectations calculated over $S_n$.

\begin{remark}[\textbf{Statistics on cosets of Young subgroups}]
The calculation involved in the evaluation $\chi(f)$ is entirely combinatorial, and can be considered independently from polynomial counting problems. In \S\ref{subsec:young_cosets} we phrase this as an independent combinatorial result, which might be of interest in other contexts. Consider the following: let $H_{\lambda}\leq S_d$ be a Young subgroup and let $gH_{\lambda}$ be a coset with $g\in N(H_{\lambda})$.
\begin{question}
What is the distribution of cycle types of permutations in $gH$?
\end{question}
Theorem \ref{thm:average_young} below answers the question and provides additional statistics on $gH_{\lambda}$.
\end{remark}

\subsection{Acknowledgments} I would like to thank 
Benson Farb for his invaluable advice and comments throughout the writing process of this document. I further thank Jesse Wolfson for his time and help shaping this paper into its current form. I thank Vladimir Drinfeld and Melanie Wood for helping me understand how this project fits in the context of existing methods. Lastly, I would like to Jeffrey Lagarias for his many helpful comments that make the paper more readable.

\section{The Frobenius distribution trace formula}
Let $G$ be a finite group. As described above, one can extend the domain of $G$-class functions to equivariant permutations of $G$-orbits by averaging.
\begin{definition}[\textbf{Evaluating class functions on permutations}]\label{def:evaluation-on-permutation}
Let $S$ be a transitive $G$-set (possibly with non-trivial stabilizers) and let $k$ be a field of characteristic $0$. For every $G$-equivariant function $\sigma:S\mor{}S$ and every class function $\chi:G\mor{}k$ define
\begin{equation}
\chi(\sigma) := \frac{1}{|\Stab(s)|}\sum_{\substack{g\in G \\ g.s =\sigma(s)}} \chi(g)
\end{equation}
where $s\in S$ is any element and $\Stab(s)$ is its stabilizer subgroup.
\end{definition}
Note that because $\chi$ is a class function, this definition does not depend on the choice of $s$, as any other choice reduces to conjugating all elements in the sum.

For the proof of Theorem \ref{thm:intro_main} we recall the following definitions.
\begin{definition}[\textbf{$G$-equivariant sheaf}]
A $G$-action on a sheaf $\mathcal{F}$ over a $G$-space $\widetilde{X}$ is a collections of sheaf morphisms $\varphi_g : (g^{-1})^*\mathcal{F} \rightarrow \mathcal{F}$ indexed by $G$ that make the following diagrams commute for every $g,h\in G$:
$$
\xymatrix{
(g^{-1})^*(h^{-1})^*\mathcal{F} \ar[r]^{\sim} \ar[d]_{(g^{-1})^* \varphi_h} &
 (h^{-1}g^{-1})^* \mathcal{F} \ar[d]^{\varphi_{gh}} \\
  (g^{-1})^*\mathcal{F} \ar[r]_{\varphi_g} & 
\mathcal{F}
}
$$
\end{definition}
\begin{example}[\textbf{Constant sheaf with a $G$-action}] \label{ex:constant-sheaf-action}
Given a $G$-action on an abelian group $A$, construct a $G$-action on the constant sheaf $\overline{A}$ over a $G$-space $\widetilde{X}$ by defining $(g^{-1})^*\overline{A}\mor{\varphi_g}\overline{A}$ to act by $g$ on all stalks, which are canonically isomorphic to $A$.
\end{example}

Now suppose $p:\widetilde{X}\mor{}X$ is the ramified $G$-cover discussed in Theorem \ref{thm:intro_main} and let $\mathcal{F}$ be a sheaf on $\widetilde{X}$ equipped with a $G$-action. Then the push-forward $p_*\mathcal{F}$ on $X$ acquires the $G$-action
$$
p_*\mathcal{F} = (p\circ g)_*\mathcal{F} \mor{\simeq} p_* g_* \mathcal{F} \mor{\simeq} p_* (g^{-1})^* \mathcal{F} \xlongrightarrow{p_* \varphi_g} p_*\mathcal{F}
$$
which is now acting by sheaf automorphisms.

\begin{definition}[\textbf{Twisted coefficient sheaf}]
In the situation described in the previous paragraph, define \textit{the twisted coefficient sheaf} corresponding to the $G$-sheaf $\mathcal{F}$ on $\widetilde{X}$ to be the subsheaf of invariants $(p_*\mathcal{F})^G$ on $X$. We denote this sheaf by $\mathcal{F}/G$ (corresponding to $X = \widetilde{X}/G$).
\end{definition}
This construction is the sheaf analog of the Borel construction in topology: giving rise to a (flat) twisted fiber bundle from the data of a $\pi_1$-action on the fiber.
%
%

\begin{lem}[\textbf{Transfer for $\mathcal{F}/G$}]\label{lem:transfer}
If multiplication by $|G|$ is an invertible transformation on $\mathcal{F}$, there is an isomorphism
$$
\Ho^i_{c,\acute{e}t}(X_{\overline{\F}_q}; \mathcal{F}/G) \cong \Ho^i_{c,\acute{e}t}(\widetilde{X}_{\overline{\F}_q}; \mathcal{F})^G.
$$

Furthermore, since the $G$-action is Galois-equivariant, so is this isomorphism.
\end{lem}
\begin{proof}
Denote the inclusion $\mathcal{F}/G = (p_*\mathcal{F})^G \hookrightarrow p_*\mathcal{F}$ by $\iota$ and define a transfer morphism $(p_*\mathcal{F})^G \overset{\tau}{\longleftarrow} p_*\mathcal{F}$ by $\tau = \sum_{g\in G} g(\cdot)$. Clearly the composition $\tau\circ \iota$ is multiplication by $|G|$ on $(p_*\mathcal{F})^G$ and the reverse composition $\iota\circ \tau$ is $|G|$ times the projection onto the $G$-invariants of $p_*\mathcal{F}$.

Consider the induced maps on cohomology
$$
\Ho^i_c(X_{\overline{\F}_q}; \mathcal{F}/G) \overset{\tau}{\underset{\iota}{\leftrightarrows}} \Ho^i_c(\widetilde{X}_{\overline{\F}_q};\mathcal{F})
$$
Assuming multiplication by $|G|$ is an invertible transformation on $\mathcal{F}$, these maps induce the desired isomorphism. Lastly, since the $G$-action on $\mathcal{F}$ is Galois equivariant, so is $\tau$ as a sum of group elements.
\end{proof}

With this in hand, the proof of Theorem \ref{thm:intro_main} follows.
\begin{proof}[Proof of Theorem \ref{thm:intro_main}]
First, observe that by the linearity of the two sides of equation \ref{eq:intro_G-L-trace-formula}, it will suffice to prove the equality for a spanning set of class functions. In particular, it will suffice to consider only characters of $G$-representations.

If $\xi$ is a $|G|$-primitive root of unity, then $\Q[\xi]$ is a splitting field for $G$, i.e. every $G$-representation in characteristic $0$ is realized over $\Q[\xi]$ (see \cite[\S 12.3, Corrolary to Theroem 24]{Se}). Since there are only finitely many irreducible representations of $G$, and each one of those is represented by finitely many matrices with entries in $\Q[\xi]$, then for every $\ell$ excluding a finite set of primes every one of the matrix entries is an $\ell$-adic integer. Fix any $\ell$ prime to $q|G|$ and large enough to have this property, i.e. that every $G$-representation in a $\Q_\ell$-vector space is defined over $\Z_\ell$.

Suppose the class function $\chi$ is the character of a $G$-representation in a $n$-dimensional $\Q_\ell$-vector space $V$. Let $\overline{V}$ denote the constant $\ell$-adic sheaf of rank $n$ on $\widetilde{X}_{\overline{\F}_q}$, and define a $G$-action on $\overline{V}$ as described in Example \ref{ex:constant-sheaf-action}. Note that since $\Fr_q$ commutes with the $G$-action on $\widetilde{X}_{\overline{\F}_q}$, it also commutes with this $G$-action on $\overline{V}$.

The Grothendieck-Lefschetz trace formula \cite[Rapport, Theorem 3.2]{SGA} applied to the twisted sheaf $\overline{V}/G$ tells us in this case that
\begin{equation}\label{eq:GL-trace-classical}
\sum_{x\in X(\F_q)} \Tr\left((\Fr_q)_x \curvearrowright \left(\overline{V}/G\right)_x \right) = \sum_{i=0}^\infty (-1)^i \Tr\left( \Fr_q \curvearrowright \Ho^i_{c,\acute{e}t}\left(X_{\overline{\F}_q};\overline{V}/G\right) \right)
\end{equation}
The rest of the proof is rewriting this equation in the form stated by the theorem.

Starting with the left-hand side, the stalk $\left(\overline{V}/G\right)_x$ is the vector space of $G$-invariant functions on $p^{-1}(x)$, and restricting to any choice of lift $\widetilde{x}\in p^{-1}(x)$ gives an isomorphism
\begin{equation} \label{eq:restriction-iso}
\left(\overline{V}/G\right)_x \cong \left(V^{p^{-1}(x)}\right)^G \mor{r_{\widetilde{x}}} V^{H_{\widetilde{x}}}
\end{equation}
where $H_{\widetilde{x}} = \operatorname{Stab}_{\widetilde{x}}$ and $V^{H_{\widetilde{x}}}$ is the subspace of $H_{\widetilde{x}}$-invariants. Indeed, this follows immediately from Frobenius reciprocity (since $p^{-1}(x)$ is a transitive $G$-set, the representation $V^{p^{-1}(x)}$ is the coinduced module $\operatorname{coInd^G_{H_{\widetilde{x}}}}V$).

Pick an element $g_0\in G$ such that $\Fr_q(\widetilde{x})=g_0(\widetilde{x})$. Then for every $v\in V^{H_{\widetilde{x}}}$ let $s=r_{\widetilde{x}}^{-1}(v)$ be the associated $G$-invariant section, i.e. $s_{g(\widetilde{x})}=g(v)$. The following equalities hold
$$
v \overset{\left(r_{\widetilde{x}}\right)^{-1}}{\longmapsto}  s \overset{(\Fr_q)_x}{\longmapsto} s\circ \Fr_q  \overset{r_{\widetilde{x}}}{\mapsto} \left(s\circ  \Fr_q\right)_{\widetilde{x}} = s_{g_0(x)} = g_0(v)
$$
and the trace of $(\Fr_q)_{x}$ coincides with that of $g_0$ acting on $V^{H_{\widetilde{x}}}$. To compute this trace let
$$
\operatorname{P}_{H_{\widetilde{x}}} = \frac{1}{|H_{\widetilde{x}}|} \sum_{h\in H_{\widetilde{x}}} h : V \mor{} V^{H_{\widetilde{x}}}
$$
be the usual projection operator. Composing $g_0$ with $\operatorname{P}_{H_{\widetilde{x}}}$ gives a self-map on $V$ that restricts to $g_0$ on $V^{H_{\widetilde{x}}}$ and is zero on the complementary representation. Thus the trace of $g_0\circ \operatorname{P}_{H_{\widetilde{x}}}$ agrees with the trace of $g_0\curvearrowright V^{H_{\widetilde{x}}}$. On the other hand,
\begin{equation} \label{eq:local-trace}
\Tr(g_0\circ \operatorname{P}_{H_{\widetilde{x}}}) = \frac{1}{|H_{\widetilde{x}}|}\sum_{h\in H_{\widetilde{x}}} \Tr(g_0 h) = \frac{1}{|H_{\widetilde{x}}|}\sum_{h\in H_{\widetilde{x}}} \chi(g_0 h)
\end{equation}
Denoting the induced permutation $\Fr_q\curvearrowright p^{-1}(x)$ by $\sigma_x$, the above average is, by definition, the evaluation $\chi(\sigma_x)$. This shows that the left-hand side of the Grothendieck-Lefchetz trace formula (Equation \ref{eq:GL-trace-classical}) coincides with that of the formula that we are proving.

Now for the right-hand side of Equation \ref{eq:GL-trace-classical}. The transfer isomorphism of Lemma \ref{lem:transfer} gives
$$
\Ho^i_c(X_{\overline{\F}_q}; \overline{V}/G) = \Ho^i_c(\widetilde{X}_{\overline{\F}_q}; \overline{V} )^G.
$$
Recall that $\overline{V}$ is a constant sheaf, so the cohomology groups can be expressed as 
$$
\Ho^i_c(\widetilde{X}_{\overline{\F}_q}; \overline{V} ) \cong \Ho^i_c(\widetilde{X}_{\overline{\F}_q}; \Q_\ell )\otimes V.
$$
Extend scalars to $\overline{\Q}_\ell$ and decompose $\Ho^{i}_c(\widetilde{X}_{\overline{\F}_q})\otimes \overline{\Q}_\ell$ into generalized $\Fr_q$-eigenspaces
$$
\Ho^{i}_c(\widetilde{X}_{\overline{\F}_q};\Q_\ell)\otimes \overline{\Q}_\ell = \oplus_{\lambda} \Ho^{i}_c(\widetilde{X}_{\overline{\F}_q})_\lambda.
$$
Since $\Fr_q$ commutes with the $G$-action, the same decomposition holds after tensoring with $V$ and restricting to the $G$-invariant subrepresentation:
$$
(\Ho^i_c(\widetilde{X}_{\overline{\F}_q}; \Q_\ell )\otimes V)^G = \oplus_{\lambda} (\Ho^{i}_c(\widetilde{X}_{\overline{\F}_q})_\lambda\otimes V)^G.
$$
Therefore the trace of $\Fr_q$ is $ \sum_{\lambda} \lambda \dim (\Ho^{i}_c(\widetilde{X}_{\overline{\F}_q})_\lambda\otimes V)^G $. The theorem follows since
$$
\dim\left( {\Ho^{i}_c(\widetilde{X}_{\overline{\F}_q})_\lambda}\otimes V\right)^G = \langle \Ho^{i}_c(\widetilde{X}_{\overline{\F}_q})^*_\lambda, V \rangle_G
$$
and $V$ was chosen so that its character is $\chi$.
\end{proof}

\section{Example: the space of polynomials}\label{sec:polys}
Consider the Vieta map $v:\A^N\rightarrow\A^N$ given by sending an $N$-tuple $(z_1,\ldots,z_N)$ of geometric points in $\A^1$ to the (coefficients of the) unique monic polynomial that has precisely these roots including multiplicity, i.e.
$$
v(z_1,\ldots,z_N) = f(t)=\prod_{i=1}^{N}(t-z_i).
$$
Denote the space of monic degree $N$ polynomials by $\Poly^N$. This space is again $\A^N$, parametrized by the polynomials' coefficients, and the coordinates of the Vieta map $v$ are given by the elementary symmetric polynomials. Thus the map $v:\A^N\mor{}\A^N$ is a ramified $S_N$ cover befitting the context of Theorem \ref{thm:intro_main}, where $S_N$ acts on the domain by permuting the coordinates.

By considering $S_N$-invariant subvarieties $\widetilde{X}\subset \A^N$ cut out by various constraints, Theorem \ref{thm:intro_main} can be used to compute statistics of spaces of polynomials whose roots are subject to the same constraints. Example of such constraints include the space of polynomials with
\begin{itemize}
\item root multiplicity bounded by some fixed $k$;
\item all roots colinear; etc.
\end{itemize}

However, to get concrete information out of Theorem \ref{thm:intro_main} one must be able to evaluate $\chi(f)$ for every $S_N$-class function $\chi$ and every polynomial $f\in \Poly^N(\F_q)$. In particular, if $f(t)$ has multiple roots then it is a ramification point of the Vieta map, and the value $\chi(f)$ is an average over some coset in $S_N$. The current section is concerned with computing the evaluations $\chi(f)$ precisely.

\subsection{Evaluation of $\chi(f)$}\label{subsec:evaluating_chi}
For computing $\chi(f)$ on every $S_N$-class function $\chi$, it suffices to compute them on the following convenient spanning set of class functions.
\begin{definition}[\textbf{Character polynomials}]
For every $k\in\N$ let $X_k:S_N\mor{}\N$ be the cycle-counting function
$$
X_k(\sigma) = \# \text{ of $k$-cycles in }\sigma.
$$
A \underline{character polynomial} is any $P\in \Q[X_1,X_2,\ldots]$. Every character polynomial $P$ gives rise to class functions $P:S_N\rightarrow \Q$, which will also be denoted by $P$. Note that $X_k \equiv 0$ whenever $k>N$.

Furthermore, for every $k,\mu_k\in \N$ define a character polynomial
$$
\binom{X_k}{\mu_k}=\frac{1}{\mu_k!}X_k(X_k-1)\ldots(X_k-\mu_k+1).
$$
More generally, for every multi-index $\mu = (\mu_1,\mu_2,\ldots)$ define its \emph{norm} $\left\| \mu \right\| = \sum_{i=k}^{\infty} k\mu_k $ and the character polynomial
$$
\binom{X}{\mu} = \binom{X_1}{\mu_1}\binom{X_2}{\mu_2}\ldots.
$$
\end{definition}
Note that $\binom{X}{\mu}=0$ unless $\mu_k=0$ for all $k>N$, so a non-zero product of this form is necessarily finite.

\begin{note}
The class function $\binom{X}{\mu}:S_N\mor{} \Q$ is counting, for every $\sigma\in S_N$, the number of ways to choose $\mu_k$ disjoint $k$-cycles in $\sigma$ for all $k$ simultaneously. Note that when $\| \mu \| = N$ there is at most one way to arrange the cycles of $\sigma$ in this way, so $\binom{X}{\mu}$ is the indicator function of the conjugacy class $C_\mu$ specified by having exactly $\mu_k$ many $k$-cycles for every $k$. For this reason it follows that the functions $\binom{X}{\mu}$ with $\|\mu \|=N$ form a basis for the class functions on $S_N$. Lastly, if $\| \mu \| > N$ then there are not enough disjoint cycles in $\sigma$, so $\binom{X}{\mu}\equiv 0$.
\end{note}

\bigskip Theorem \ref{thm:intro_evaluating_xmu} describes the evaluation $\binom{X}{\mu}(f)$ for every polynomial $f\in \Poly^N(\F_q)$. The statement involves the divisibility of $f$ by other polynomials $g$, and its formulation used the division symbols $\epsilon_g$, recalled below.
\begin{notation}[\textbf{The algebra of division symbols}]
For every monic polynomial $g\in \F_q[t]$ introduce a formal symbol $\epsilon_g$ that measures divisibility by $g$ in the following way: for every polynomial $f\in \F_q[t]$ set
\begin{equation}\label{eq:evaluation_on_epsilons}
\epsilon_g(f) =  \begin{cases}
1 \quad & g|f \\ 0 \quad & \text{otherwise}.
\end{cases}
\end{equation}
Let $R_q$ be the free $\Q$-vector space spanned by these $\epsilon_g$ symbols with $g$ ranging over all monic polynomials in $\F_q[t]$. Extend the evaluation maps $f\mapsto \epsilon_g(f)$ linearly to $R_q$.

Furthermore define multiplication on the symbols $\epsilon_g$ by
\begin{equation}
\epsilon_g \cdot\epsilon_h = \epsilon_{gh}.
\end{equation}
turning $R_q$ into a $\Q$-algebra. One should take care and observe that the evaluation $f\mapsto \epsilon_g(f)$ is \textit{not} multiplicative on $R_q$. In fact, with respect to the evaluation on any $f\in \F_q[t]$, every element $\epsilon_g$ is nilpotent. This nilpotence property turns out to be an essential part of Theorem \ref{thm:intro_evaluating_xmu}.
\end{notation}

One can evaluate all functions $\binom{X}{\mu}$ simultaneously using a generating function: set 
$$
F(\textbf{t})= F(t_1,t_2,\ldots) := \sum_{\mu=(\mu_1,\mu_2,\ldots)} \binom{X}{\mu} t_1^{\mu_1}t_2^{\mu_2}\ldots
$$
This series evaluates on a polynomial $f\in \Poly^N(\F_q)$ term-wise, that is:
$$
F(\textbf{t})_{(f)} = \sum_{\mu}\, \binom{X}{\mu}{}_{(f)} \; t_1^{\mu_1}t_2^{\mu_2}\ldots
$$
and the resulting series is generated by that coefficients that we wish to compute. Theorem \ref{thm:intro_evaluating_xmu} then provides an explicit description of all evaluations by
\begin{equation}
F(\textbf{t})_{(f)} = \exp\left( \sum_{k=1}^\infty \sum_{p\in \Irr_{|k}}\deg(p)\epsilon_{p}^{\frac{k}{\deg(p)}} \, \frac{t_k}{k} \right){(f)}
\end{equation}
with $\Irr_{|k}$ being the set of monic irreducible polynomials of degree dividing $k$. Recall that to evaluate an element of $R_q$ on $f$ one must first expand any product into monomials in the $\epsilon_g$ symbols, and then evaluate according to the divisibility of $f$ by $g$.

\bigskip We illustrate how to use this result in a couple of examples.
\begin{example}[\textbf{Square-free polynomials}]
Consider the special case where $f\in \Poly^N$ is square-free. This case is simple since such $f$ are unramified points of the Vieta map, and furthermore one does not encounter the non-multiplicative behavior of the evaluation on $\epsilon_g$ symbols.

One the one hand, since $f$ is an unramified point, the Frobenius permutation on the roots determines an element $g_f\in S_N$ unique up to conjugation, and by definition $\chi(f)=\chi(g_f)$. In particular, the evaluation is multiplicative in $\chi$. Elementary Galois theory shows that for $\chi = X_k$ the value $X_k(f)$ is the number of degree $k$ irreducible factors of $f$. Using these facts one can easily write down a formula for $\binom{X}{\mu}(f)$ that does not go though Theorem \ref{thm:intro_evaluating_xmu}. However, the point of this example is to see how to use Equation \ref{eq:evaluation_on_xmu} in calculations, so we shall ignore this argument.

For every irreducible polynomial $p(t)$, evaluating the symbol $\epsilon_{p^r}$ on $f$ would give $0$ whenever $r>1$, as $f$ will not be divisible by such powers of $p$. Thus in Equation \ref{eq:evaluation_on_xmu}, the only contributions to the sum over $p\in \Irr_{|k}$ come from $p$ of degree $k$ precisely. We can therefore simplify the expression and get that on the set of square-free polynomials the following two functions coincide:
\begin{equation}\label{eq:square_free_X_mu}
\binom{X}{\mu} = \prod_{k=1}^\infty \frac{1}{\mu_k !} (\sum_{p\in \Irr_{=k}(\F_q)} \epsilon_p)^{\mu_k}
\end{equation}
where $\Irr_{=k}(\F_q)$ is the set of irreducible polynomials of degree \emph{equal} to $k$ over $\F_q$. Observe the following two properties of this product.
\begin{enumerate}
\item If $p$ and $q$ are coprime polynomials then $\epsilon_{p} (f) \cdot \epsilon_q (f) = \epsilon_{pq}(f)$ for every $f$. Therefore, since the $k$-th term of the product in Equation \ref{eq:square_free_X_mu} involves only irreducible polynomials of degree $k$, the different terms evaluate on $f$ multiplicatively, i.e.
$$
\left[\prod_{k=1}^\infty \frac{1}{\mu_k !} (\sum_{p\in \Irr_{=k}(\F_q)} \epsilon_p)^{\mu_k}\right] (f) = \prod_{k=1}^\infty \left[\frac{1}{\mu_k !} (\sum_{p\in \Irr_{=k}(\F_q)} \epsilon_p)^{\mu_k} (f)\right]
$$
and thus each term may be evaluated separately.
\item For each $k$, use again the fact that any power $(\epsilon_p)^r=\epsilon_{p^r}$ evaluates to $0$ on a square-free polynomial whenever $r>1$. Expanding the $\mu_k$-th power and eliminating high powers of $\epsilon_p$'s
$$
\frac{1}{\mu_k !}(\sum_{p\in \Irr_{=k}(\F_q)} \epsilon_p)^{\mu_k} = \sum_{\{p_1,\ldots,p_{\mu_k}\}\subset \Irr_{=k}(\F_q)} \epsilon_{p_1\ldots p_{\mu_k}}
$$
where the sum goes over all sets of degree-$k$ irreducibles of cardinality $\mu_k$. Since the $\epsilon$ symbols evaluate to either $1$ or $0$ on $f$, it follows that the sum evaluates to \textit{the number of ways to choose $\mu_k$ distinct irreducible factors of $f$  with degree $k$}.
\end{enumerate}
\begin{cor}
If $f$ is a square-free polynomial then
$$
\binom{X}{\mu}(f) = \# \text{ ways to choose $\mu_k$ many degree $k$ irreducible factors of $f$ for every $k$. }
$$
\end{cor}
\end{example}

\begin{example}[\textbf{Degree $1$ character polynomials}]
The evaluation of $X_k$ is most straightforward. This is the expression $\binom{X}{\mu}$ with $\mu_k=1$ and $\mu_j=0$ for all $j\neq k$. In this case Equation \ref{eq:evaluation_on_xmu} simplifies to
$$
X_k = \frac{1}{k} \sum_{p\in \Irr_{|k}} \deg(p)\epsilon_{p}^{\frac{k}{\deg(p)}} = \sum_{d|k} \frac{d}{k}\sum_{p\in \Irr_{=d}} \epsilon_{p^{k/d}}.
$$
When evaluating this expression on a polynomial $f$, since $\epsilon$-symbols evaluate to either $0$ or $1$, the sum becomes a simple count:
\begin{equation}\label{eq:X_k_on_f}
X_k(f) = \sum_{d|k} \frac{d}{k} \# \{\text{ irreducible factors of degree $d$ that divide $f$ at least $k/d$ times }\}.
\end{equation}

Considering the two extreme cases: if $f$ is square-free this reduces back to the count of degree $k$ irreducible factors; and if $f = p(t)^r$ with $p$ irreducible of degree $d$ then $X_{d\cdot \ell}(f)=\frac{1}{\ell}$ if $r\geq\ell$ and $0$ otherwise.

\begin{remark}
The origin of the sum in Equation \ref{eq:X_k_on_f} is clear when one considers the stack quotient $[\A^N/S_N]$: the Vieta map factors though the universal map from the quotient stack to the quotient variety
$$
[A^N/S_N] \mor{} A^N/S_N = \Poly^N
$$
and the fiber of this map over $f$ contains multiple points that each contribute a term to $X_k(f)$. At the same time, the points on the stack have automorphisms which account for the denominators.
\end{remark}
\end{example}

\begin{proof}[Proof of Theorem \ref{thm:intro_evaluating_xmu}]
Let $f(t)$ be a monic polynomial over $\F_q$. Suppose $f$ decomposes as
$$
f= p_1(t)^{r_1}\ldots p_n(t)^{r_n}
$$
where the $p_i(t)$'s are the distinct irreducible factors of $f$ and set $d_i=\deg(p_i)$. Over the algebraic closure $\overline{\F}_q$ every factor $p_i(t)$ decomposes further as a product of linear terms
$$
p_i(t)= (t-\alpha_{i,1})(t-\alpha_{i,2})\ldots (t-\alpha_{i, d_i})
$$
with all $\alpha_{i,k}$ distinct. Thus $f(t)$ is the product
$$
f(t) = \prod_{\substack{1\leq i \leq n\\ 1\leq k \leq d_i}}(t-\alpha_{i,k})^{r_i}.
$$
and the degree of $f$ is $N=\sum_{i=1}^n d_i r_i$. 

Under the Vieta map $v:\A^N \mor{} \A^N$
$$
(z_1,\ldots,z_N) \mapsto p(t)=(t-z_1)\ldots(t-z_N)
$$
the polynomial $f(t)$ is the image of the $N$-tuple
$$
\alpha_f := (\underbrace{\alpha_{1,1},\ldots,\alpha_{1,1}}_{r_1\text{ times}}, \underbrace{\alpha_{1,2},\ldots,\alpha_{1,2}}_{r_1\text{ times}},\ldots,
\underbrace{\alpha_{1,d_1},\ldots,\alpha_{1,d_1}}_{r_1\text{ times}}, \underbrace{\alpha_{2,1},\ldots,\alpha_{2,1}}_{r_2\text{ times}},\ldots).
$$
Under the $S_N$-action of permuting the coordinates, the stabilizer of $\alpha_f$ is the Young subgroup 
$$
H_f := \underbrace{S_{r_1}\times \ldots \times S_{r_1}}_{d_1 \text{ times}}\times \ldots \times \underbrace{S_{r_n}\times \ldots \times S_{r_n}}_{d_n \text{ times}} = (S_{r_1})^{d_1}\times \ldots (S_{r_n})^{d_n}.
$$
For clarity of notation, relabel the points of the set $\{1,\ldots, N\}$ by choosing a bijection
$$
T_f = \{ (i,j,k) : 1\leq i\leq n, \, 1\leq j\leq r_i,\, k\in \Z/d_i\Z \} \mor{\cong} \{1,\ldots,N\}
$$
and thinking of $S_N$ as the symmetry group of $T_f$. One should think of the $3$-tuple $(i,j,k)\in T_f$ as corresponding to the root $\alpha_{i,k}$ of the $j$-th copy of $p_i(t)$ (that is, $i$ indexes the irreducible factor, $j$ indexes which of the multiple copies of $p_i$ one is considering, and $k$ indexes the roots of $p_i$).

The Frobenius $\Fr_q$ acts on the roots of $f(t)$ by the cyclic permutation $\alpha_{i,k}\mapsto \alpha_{i,k+1}$, where the second subscript $k$ belongs to $\Z/d_i\Z$. Thus the $\Fr_q$-action is lifted by the permutation $\tau\in S_N$ given by $\tau(i,j,k)=(i,j,k+1)$. 

\bigskip Our task is now to compute the value of $\chi(\sigma_f)$ for every class function $\chi$, i.e. the average value of $\chi$ on the coset $\tau H_f$. Define a new function $\chi^{(\tau)}:S_N\mor{}\Q$ by
$$
\chi^{(\tau)}(h) = \chi(\tau \cdot h).
$$
Then it is clear that
$
\displaystyle{\chi(f) = \frac{1}{|H_f|}\sum_{h\in H_f} \chi^{(\tau)}(h) = \mathbb{E}_{H_f}[\chi^{(\tau)}]}
$
and that the operation $\chi \mapsto \chi^{(\tau)}$ commutes with arithmetic operations on functions. Thus to evaluate $f$ on arbitrary character polynomials, one can start by understanding the function $X_k^{(\tau)}$. For this purpose introduce the following notation.

\begin{definition}[\textbf{The projection operator $m_{i_0}$}]
For every $1\leq i_0 \leq n$ (i.e. for every irreducible factor $p_{i_0}|f$) define a projection operator $m_{i_0}: H_f \mor{} S_{r_{i_0}}$ as follows.

An element of $H_f=\prod_{i,k}S_{r_i}$ is given by a sequence of permutations $(h_{i,k})_{i,k}$ where $h_{i,k}\in S_{r_i}$. Define
$$
m_{i_0} : (h_{i,k})_{i,k} \mapsto h_{i_0,d_i}\cdot h_{i_0,d_i-1}\cdot  \ldots\cdot  h_{i_0,2}\cdot h_{i_0,1}.
$$
\end{definition}

\begin{claim} \label{claim:X_tau}
For every $r\in \N$ the function $X_r^{(\tau)}$ can be presented as a sum
\begin{equation}
X_r^{(\tau)} = \sum_{\substack{1\leq i \leq n \\ d_i \mid r}} X_{\frac{r}{d_i}} \circ m_i.
\end{equation}
\end{claim}
\begin{proof}
Fix an element $h=(h_{i,k})_{i,k}\in H_f$. The value $X_r(\tau h)$ is the number of $r$-cycles the permutation $\tau h\in \tau H_f$ has when acting on the set $T_f$. The action of this permutation on an element $(i,j,k)\in T_f$ is by
$$
(i,j,k) \overset{h}{\mapsto} (i,h_{i,k}(j),k) \overset{\tau}{\mapsto} (i,h_{i,k}(j),k+1).
$$
Observe that the $i$-value is fixed by this action. Thus every cycle of $\tau h$ has a well-defined $i$-value, and there is a decomposition
$$
X_r^{(\tau)} = \sum_{i=1}^n X_r^{(i)}
$$
where $X_r^{(i)}$ counts only the number of $r$-cycles with $i$-value equal to $i$. It therefore remains to show that $X_r^{(i)} = X_{\frac{r}{d_i}} \circ m_i$ if $d_i \mid r$ and $0$ otherwise.

Fix $i$. An element $(i,j,1)\in R_f$ belongs to an $r$-cycle if it is fixed by $(\tau h)^r$ and not by any smaller power of $\tau \cdot h$. Compute
$$
(\tau h)^m (i,j,1) = (i, (h_{i,m}\cdot \ldots\cdot  h_{i,2} \cdot h_{i,1})(j) , m+1)
$$
so for $(i,j,1)$ to be fixed, demand that $m+1 \equiv 1 \mod{d_i}$, i.e. that $d_{i} \mid m$. Restricting to this case, write $m=\ell \cdot d_i$. Now the $j$-value after applying $(\tau h)^m$ is
$$
(h_{i,m}\cdot \ldots\cdot h_{i,1})(j) = (h_{i,d_i}\cdot \ldots\cdot  h_{i,1})^\ell(j) = m_{i}(h)^{\ell}(j)
$$
This shows that $(i,j,1)$ belongs to an $r$ cycle if and only if $d_i\mid r$ and $j$ belongs to an $\frac{r}{d_i}$-cycle of $m_i(h)$.

Since every orbit of $\tau\cdot h$ includes an element of the form $(i,j,1)$ (as any triple $(i,j,k)$ goes to an element of the form $(i, j', 1)$ after $(d_i-k+1)$ applications of $\tau h$), one only need to count the number of such elements that belong to $r$-cycles. By the previous paragraph, the number of $j$'s for which $(i,j,1)$ belong to $r$-cycle is equal to the number of $\frac{r}{d_i}$-cycles of $m_i(h)$ when $d_i\mid r$, and that otherwise there are none. Thus the proclaimed equality follows
$$
X^{(i)}_r(\tau h) = \begin{cases} X_{\frac{r}{d_i}} (m_i(h)) \quad & d_i \mid r \\ 0 & \text{ otherwise}
\end{cases}.
$$
\end{proof}

We are now ready to compute $\binom{X}{\mu}(f)$ for every multi-index $\mu$. Our calculation proceeds using the generating function introduced above. Note that for every fixed $N$, the class functions $X_k$ with $k>N$ are identically $0$, so the function $F(\textbf{t})$ is really a polynomial in the variables $t_1,\ldots,t_N$. Throughout the proceeding calculation one should remember that all expressions involved are really finite.

Apply the operation $\chi \mapsto \chi^{(\tau)}$ to $F$ termwise: $F^{(\tau)}(\textbf{t}) = \sum_{\mu}\binom{X}{\mu}^{(\tau)}\textbf{t}^\mu$. Using the observation of Claim \ref{claim:X_tau} it follows that
\begin{eqnarray}
F^{(\tau)} = \prod_{k=1}^{\infty} (1+t_k)^{\sum_{\{i: d_i|k\}}^n X_{k/d_i}\circ m_i} =  \prod_{k=1}^\infty\prod_{\{i: d_i|k\}}^n (1+t_k)^{X_{k/d_i}\circ m_i} = \prod_{i=1}^\infty \prod_{\ell=1}^{\infty}(1+t_{d_i\ell})^{X_\ell\circ m_i}.
\end{eqnarray}
where in the last equality we relabeled $k=\ell\cdot d_i$ to include only pairs $(i,k)$ in which $d_i|k$. One now notices that the function $F^{(\tau)}$ factors though the product of projections
$$
m = (m_1,\ldots,m_n): H_f \mor{} S_{r_1}\times\ldots\times S_{r_n}.
$$
The next observation simplifies the problem greatly.
\begin{claim}
The product of projections
$$
m = (m_1,\ldots,m_n): H_f \mor{} S_{r_1}\times\ldots\times S_{r_n}
$$
is precisely an $\frac{|H_f|}{|S_{r_1}\times\ldots\times S_{r_n}|}$-to-$1$ function. In other words, the map $m$ is measure preserving between the two uniform probability spaces.
\end{claim}
\begin{proof}
Fix any element $(\sigma_1,\ldots,\sigma_n)\in S_{r_1}\times\ldots\times S_{r_n}$. Then for every $i$ and every choice of elements $(h_{i,1},\ldots,h_{i,d_i-1})\in S_{r_i}^{d_i-1}$ there exists a unique $h_{i,d_i}$ that satisfies the equality
$$
h_{i,d_i}\cdot h_{i,d_i-1}\cdot \ldots \cdot h_{i,1} = \sigma_i
$$
namely $h_{i,d_i}= \sigma_i\cdot (h_{i,d_i-1}\cdot \ldots \cdot h_{i,1})^{-1}$. Since the terms with different index $i$ do not appear in this expression, they may be chosen independently.

It follows that there is a bijection $\prod_{i=1}^n S_{r_i}^{d_i-1} \cong m^{-1}(\sigma_1,\ldots,\sigma_n)$, thus demonstrating the claim.
\end{proof}
The fact that $F^{(\tau)}$ factors through the measure preserving map $m$ allows one to compute the $H_f$-expected value by computing it on $\prod_{i=1}^n S_{r_i}$ instead:
\begin{equation}
\mathbb{E}_{H_f}[F^{(\tau)}\circ m] = \mathbb{E}_{\prod_{i=1}^n S_{r_i}}[F^{(\tau)}] = \prod_{i=1}^n \mathbb{E}_{S_{r_i}}[\prod_{\ell=1}^{\infty}(1+t_{d_i\ell})^{X_{\ell}}].
\end{equation}

The derivation of our formula follows from the following key observation.
\begin{thm}\label{thm:CEF_expected_value}
Fix $r\in \N$ and let $\epsilon$ be a formal nilpotent element such that $\epsilon^{r+1}=0$ but $\epsilon^r\neq 0$. Then for every $d\in \N$ there is an equality
\begin{equation}
G_{d,r}(\textbf{t},\epsilon):=\mathbb{E}_{S_{r}}[\prod_{\ell=1}^{\infty}(1+\epsilon^{\ell} t_{d\ell})^{X_{\ell}}] = \exp\left( \sum_{\ell}\epsilon^{\ell}\frac{t_{d\ell}}{\ell} \right).
\end{equation}

Furthermore, the introduction of $\epsilon$ to the left-hand side of the equation does not cause any loss of information in the sense that replacing every nonzero power of $\epsilon$ by $1$ recovers the expectations. Formally, denoting the ring $\Q[t_1,t_2,\ldots]$ by $A$, the $A$-module map $A[\epsilon]/(\epsilon^{r+1}) \mor{\phi_\epsilon} A$ defined by $\epsilon^j\mapsto 1$ for all $j\leq r$ sends
$$
\phi_\epsilon: G_{d,r}(\textbf{t},\epsilon)\mapsto G_{d,r}(\textbf{t},1).
$$
The latter function is the generating function of the expectations $\mathbb{E}_{S_r}\left[\binom{X_{\ell/d}}{\mu}\right]$(defined to be $0$ unless $d|\ell$).
\end{thm}
\begin{proof}
Expand the left-hand side
\begin{equation}
\prod_{\ell=1}^{\infty}(1+\epsilon^{\ell} t_{d\ell})^{X_{\ell}} = \sum_{\mu} \binom{X}{\mu}\epsilon^{1\mu_1+2\mu_2+\ldots}t_d^{\mu_1}t_{2d}^{\mu_2}\ldots = \sum_{\mu}\binom{X}{\mu}\epsilon^{\|\mu\|}t_d^{\mu_1}t_{2d}^{\mu_2}\ldots.
\end{equation}
Recall that if $\|\mu \| > r$ then $\binom{X}{\mu}\equiv 0$, so setting $\epsilon^i=1$ for all $i\leq r$ is the same as having $1$'s in place of $\epsilon$ everywhere.

Use the following calculation of \cite{CEF-pointcounts}.
\begin{fact}[ In the proof of {\cite[Proposition 3.9]{CEF-pointcounts}} ]\label{fact:expectation}
If $\mu$ is a multi-index with $\|\mu\|\leq r $ then 
$$
\mathbb{E}_{S_r}\left[ \binom{X}{\mu} \right] = \prod_{\ell=1}^{\infty} \frac{1}{\ell^{\mu_\ell}\mu_\ell !}.
$$
\end{fact}
When multiplying this equation by $\epsilon^{\| \mu \|}$ one gets an equality that holds for all $\mu$. Thus when evaluating the expectation on the generating function
$$
\mathbb{E}_{S_r}\left[ \sum_{\mu}\binom{X}{\mu}\epsilon^{\|\mu\|}t_d^{\mu_1}t_{2d}^{\mu_2}\ldots \right] = \sum_{\mu}\frac{(\epsilon t_d)\mu_1}{1^{\mu_1}\mu_1!}\frac{(\epsilon^2 t_{2d})\mu_2}{2^{\mu_2}\mu_2!}\ldots = \prod_{\ell=1}^{\infty}\sum_{\mu_\ell=1}^{\infty}\frac{1}{\mu_\ell !}\left(\frac{\epsilon^\ell t_{d\ell}}{\ell}\right)^{\mu_\ell}
$$
$$
= \prod_{\ell=1}^{\infty}\exp\left( \frac{\epsilon^\ell t_{d\ell}}{\ell} \right).
$$
To get the stated form of this expression, use the multiplicative property of the exponential series.
\end{proof}

Apply this observation to the calculation of $\mathbb{E}_{H_f}[F^{(\tau)}]$: introduce $n$ nilpotent elements $\epsilon_i$ with respective order $r_i+1$ and let $\phi$ be the map $\Q[\textbf{t}]$-linear map that sends $(\epsilon_i)^j\mapsto 1$ whenever $j\leq r_i$.
Then
$$
\mathbb{E}_{H_f}\left[ F^{(\tau)}(\textbf{t}) \right] = \prod_{i=1}^n \mathbb{E}_{S_{r_i}}[\prod_{\ell=1}^{\infty}(1+t_{d_i\ell})^{X_{\ell}}] = \phi \left[\prod_{i=1}^{n}\exp\left( \sum_{\ell}\epsilon_i^{\ell}\frac{t_{d_i\ell}}{\ell} \right)\right] = \phi \left[\exp\left( \sum_{i=1}^{n}\sum_{\ell}\epsilon_i^{\ell}\frac{t_{d_i\ell}}{\ell} \right)\right].
$$
Relabel the terms $d_i\ell = k$ back by summing only over $\{i: d_i|k\}$ and replacing $\ell = k/d_i$. The resulting expression becomes
\begin{equation}\label{eq:final_form_for_young_cosets}
\mathbb{E}_{H_f}\left[ F^{(\tau)}(\textbf{t}) \right] = \phi \left[\exp\left( \sum_{k=1}^\infty \sum_{\{i: d_i|k\}}\epsilon_i^{k/d_i}d_i \, \frac{t_k}{k} \right)\right].
\end{equation}

To bring the expression to the desired form, replace the $\epsilon_i$'s with the symbols $\epsilon_p\in R_q$ introduced above. Recall that the $(S_{r_i})^{\times d_i}$-factor of $H_f$ corresponds to the irreducible factor $p_i$ of degree $d_i$ that divides $f$ precisely $r_i$ times. Thus, using the symbol $\epsilon_{p_i}$ and its evaluation of $f$ as defined in Equation \ref{eq:evaluation_on_epsilons}, there is an equality
$$
\phi(\epsilon_i^d) = \epsilon_{p_i}^d(f) = \begin{cases}
1 \quad d\leq r_i \\ 0 \quad \text{otherwise}
\end{cases}.
$$
More generally, $\phi \left[\epsilon_{1}^{j_1}\ldots \epsilon_n^{j_n}\right]=1$ if $j_i\leq r_i$ for all $i$ and is $0$ otherwise. This is precisely the value of $\epsilon_{p_1}^{j_1}\ldots \epsilon_{p_n}^{j_n} = \epsilon_{p_1^{j_1}\ldots p_n^{j_n}}$ evaluated on $f$, since the $p_i$'s are coprime. It follows that Equation \ref{eq:final_form_for_young_cosets} can be written with $\epsilon_{p_i}$ in the place of $\epsilon_i$ everywhere.

For every other irreducible polynomial $p\neq p_1,\ldots,p_n$ the evaluation $\epsilon_p(f)=0$, so adding all such symbols into Equation \ref{eq:final_form_for_young_cosets} does not change the resulting evaluation. It does, however, allow us to write the sum uniformly without making any reference to the divisors of $f$. This is the sought after form of the generating function.

Lastly, to get the individual expectations $\binom{X}{\mu}(f)=\mathbb{E}_{H_f}\left[ \binom{X}{\mu}^{(\tau)} \right]$ one needs only to look compute the $\mu$-th partial derivative with respect to $\textbf{t}$. This resulting expression is as stated.
\end{proof}

\subsection{The case of all monic polynomials} \label{subsec:all_polynomials}
We now complete the calculation in the case where one does not impose any restrictions on roots, i.e. we are considering the $S^N$ action on the whole of $\A^N$. The cohomology groups of $\A^N$ are very simple: they are known to be $\Q_\ell(0)$ in dimension $0$ and vanish in all higher degrees. Moreover, the induced $S_N$-action on these cohomology groups is trivial. Thus Theorem \ref{thm:intro_main} reduces to the surprising Corollary \ref{cor:intro_equal_expectations}.
\begin{proof}[Proof of Corollary \ref{cor:intro_equal_expectations}]
For every $S_N$-class function $\chi$ Theorem \ref{thm:intro_main} gives an equality
$$
\frac{1}{q^N}\sum_{f\in \Poly^N(\F_q)} \chi(\sigma_f) = \langle H^0(\A^N;\Q_\ell), \chi\rangle_{S_N} = \frac{1}{N!}\sum_{g\in S_N} 1\cdot \chi(g).
$$
\end{proof}
\begin{remark}
In Corollary \ref{cor:intro_equal_expectations}, there is no reason to restrict attention to the group action $S_N \curvearrowright \A^N$: one can more generally consider any subgroup of $\Aut(\A^N)(\F_q) = \operatorname{Aff}_N(\F_q)$ and get a similar result:
\begin{thm}[\textbf{Equal expectation for $G\leq \operatorname{Aff}_N(\F_q)$}]
Let a subgroup $G\leq \operatorname{Aff}_N(\F_q)$ act on $\A^N$ naturally, and denote the resulting quotient $\A^N/G$ by $X_G$. Then every $G$-class function $\chi$ induces random variables on the two uniform probability spaces $X_G(\F_q)$ and $G$, and these satisfy
$$
\frac{\left|X_G(\F_q)\right|}{q^N}\mathbb{E}_{X_G(\F_q)}[\chi] = \mathbb{E}_G[\chi].
$$
\end{thm}
We will not pursue this idea any further.
\end{remark}
\begin{example}[\textbf{Case $\chi = \binom{X}{\mu}$}]
The explicit form of $\binom{X}{\mu}(f)$ given by Theorem \ref{thm:intro_evaluating_xmu} can be used to unpack Equation \ref{eq:equal-expectation} and get definite facts regarding polynomials.
\begin{claim}
Equation \ref{eq:equal-expectation} is equivalent to the well known Necklace relations: for every degree $d$ let $N_d$ be the number of degree $d$ monic irreducible polynomials in $\F_q[t]$, then for every $k\in \N$ one has
\begin{equation}
\sum_{d|k} dN_d = q^k.
\end{equation}

More specifically, for every multi-index $\mu=(\mu_1,\mu_2,\ldots)$ the following equality holds
\begin{equation}
\mathbb{E}_{\Poly^N(\F_q)}\left[\binom{X}{\mu}\right] = \mathbb{E}_{S_N}\left[\binom{X}{\mu}\right]\prod_{k=1}^\infty \left(\frac{1}{q^k}\sum_{d|k} d N_d\right)^{\mu_k}.
\end{equation}
\end{claim}

\begin{proof}
The major challenge in computing the evaluation $\binom{X}{\mu}(f)$ directly is that one has to deal with the irreducible decomposition of $f$, and these can take many forms. It turns out that this challenge disappears completely when one in only interested in the average value over all $f\in \Poly^N(\F_q)$, as the following observation shows. 

\begin{lem}[\textbf{Average of $\epsilon_g$}]\label{lem:average_of_epsilon}
Let $\mathbb{E}$ denote the expectation over the set $\Poly^N(\F_q)$ of monic polynomials of degree $N$ as a uniform probability space. Then for every polynomial $g$, the symbol $\epsilon_g$ satisfies
\begin{equation}
\mathbb{E}[\epsilon_g] = \begin{cases}
\frac{1}{q^{\deg(g)}}\quad \deg(g)\leq N \\ 0\quad \text{Otherwise}.
\end{cases}
\end{equation}

Stated equivalently, let $\epsilon$ be a formal nilpotent element of order $N+1$. Then the $\Q$-algebra homomorphisms $R_q \mor{\Lambda} \Q[\epsilon]/(\epsilon^{N+1})$ defined by linearly extending
$$
\Lambda:\epsilon_g \mapsto \left(\frac{\epsilon}{q}\right)^{\deg(g)} \quad \forall g\in \F_q[t]
$$
and the $\Q$-module map $\phi_\epsilon:\Q[\epsilon]/(\epsilon^{N+1})\mor{}\Q$ defined by sending $\epsilon^j\mapsto 1$ for all $j\leq N$ satisfy the relation 
$$
\phi_\epsilon\circ \Lambda = \mathbb{E}.
$$

Thus the expectation factors through the homomorphism $\Lambda$, which sends all $\epsilon_g$ symbols to a single expression involving $\epsilon$.
\end{lem}
\begin{proof}
First, if $\deg(g)>N$ then $\epsilon_g\equiv 0$ on $\Poly^N$ and thus $\mathbb{E}[\epsilon_g]=0$. Next, if $\deg(g)\leq N$, then since $\F_q[t]$ is a UFD, division gives a bijection
$$
\left\{ f\in \Poly^N(\F_q) \text{ s.t. } g|f \right\} \longleftrightarrow \Poly^{(N-\deg(g))(\F_q)}.
$$
Thus the number of monic polynomial divisible by $g$ is precisely $q^{N-\deg(g)}$. Evaluating the expectation amounts to dividing this count by the cardinality of $\Poly^N(\F_q)$, which is $q^N$.
\end{proof}

Let $F(\textbf{t})$ be the generating function of the $\binom{X}{\mu}$'s as introduced in the previous section. Recall that by Theorem \ref{thm:intro_evaluating_xmu}
$$
F(\textbf{t}) = \exp\left( \sum_{k=1}^\infty \sum_{p\in \Irr_{|k}}\deg(p)\epsilon_{p}^{\frac{k}{\deg(p)}} \, \frac{t_k}{k} \right)
$$
as functions on $\Poly^N(\F_q)$. Using Lemma \ref{lem:average_of_epsilon} proved above, one can readily compute $\mathbb{E}[F(\textbf{t})]$ by first applying $\Lambda$: Since the evaluation $\Lambda:\epsilon_g \mapsto \left(\frac{\epsilon}{q}\right)^{\deg(g)}$ is a $\Q$-algebra homomorphism, it commutes with taking the exponential series
$$
\Lambda(F) = \exp\left( \sum_{k=1}^\infty \sum_{p\in \Irr_{|k}}\deg(p)  \Lambda(\epsilon_{p})^{\frac{k}{\deg(p)}} \, \frac{t_k}{k} \right) = \exp\left( \sum_{k=1}^\infty \sum_{p\in \Irr_{|k}}\deg(p)  \left(\frac{\epsilon}{q}\right)^k \, \frac{t_k}{k} \right).
$$
Collect all terms that depend only on $k$
$$
\exp\left( \sum_{k=1}^\infty \frac{\epsilon^k t_k}{k} \cdot \frac{1}{q^k}\sum_{p\in \Irr_{|k}}\deg(p) \right).
$$
Reindex the sum $\displaystyle{\sum_{p\in \Irr_{|k}}\deg(p)}$ based on the degrees $d=\deg(p)$, and observe that it is precisely the sum $\sum_{d|k} d N_d$ that appears in the Necklace relations.

By differentiation with respect to $\textbf{t}$ one finds that the $\mu$-th coefficient of this generating function is
$$
\Lambda \binom{X}{\mu} = \prod_{k=1}^\infty \frac{\epsilon^{k\mu_k}}{k^{\mu_k} \mu_k!} \cdot \left(\frac{1}{q^k}\sum_{d|k} d N_d\right)^{\mu_k} = \epsilon^{\|\mu \|} \prod_{k=1}^\infty \frac{1}{k^{\mu_k} \mu_k!} \cdot \left(\frac{1}{q^k}\sum_{d|k} d N_d\right)^{\mu_k}.
$$
Furthermore, Fact \ref{fact:expectation} produced an equality
$$
\phi_\epsilon\left(\epsilon^{\|\mu \|} \prod_{k=1}^\infty \frac{1}{k^{\mu_k} \mu_k!}\right) = \mathbb{E}_{S_N}\left[ \binom{X}{\mu} \right]
$$
which produces the desired result using $\phi_\epsilon\circ \Lambda = \mathbb{E}_{\Poly^N(\F_q)}$.
\end{proof}
\end{example}

\subsection{Statistics on cosets of Young subgroups}\label{subsec:young_cosets}
One can rephrase Theorem \ref{thm:intro_evaluating_xmu} as a purely combinatorial statement regarding the cycle-decomposition statistics of cosets of Young subgroups. This could be stated as follows.

As stated above, for every multi-index $\mu=(\mu_1,\mu_2,\ldots)$, the character polynomial $\binom{X}{\mu}$ counts the number of ways to arrange cycles into sets of $\mu_k$ many $k$-cycles for every $k$.
\begin{thm}[\textbf{Statistics on Young cosets}]\label{thm:average_young}
Let $H = S_{\lambda_1}\times \ldots \times S_{\lambda_m}$ be a Young subgroup of $S_N$. Consider a coset $gH$ where $g\in S_N$ is in the normalizer of $H$ (i.e. $gH=Hg$). Then conjugation by $g$ induces a permutation $\tau$ of the $S_{\lambda_i}$ factors, say
$$
H = S_{r_1}^{d_1}\times \ldots \times S_{r_n}^{d_n}
$$
and $\tau$ cyclically permutes the factors in each $S_{r_i}^{d_i}$.

Then for every multi-index $\mu=(\mu_1,\mu_2,\ldots)$, the expected value of $\binom{X}{\mu}$ on the coset $gH$ is given by the expression
\begin{equation} \label{eq:young_statistics}
\mathbb{E}_{S_N}\left[\binom{X}{\mu}\right]\Phi\left[\prod_{k=1}^{\infty}\left(\sum_{\{i: d_i|k\}}\epsilon_i^{\frac{k}{d_i}}d_i\right)^{\mu_k}\right]
\end{equation}
where the symbol $\epsilon_i$ is a formal nilpotent element of order $r_i+1$, and one eliminates these applying the $\Q$-linear transformation
$$
\Phi:\Q[\epsilon_1,\ldots,\epsilon_n]/(\epsilon_i^{r_i+1}) \mor{} \Q \quad, \qquad \Phi: (\epsilon_i)^j \mapsto 1 \quad \forall j\leq r_i.
$$

In particular, for every cycle type $\mu=(\mu_1,\mu_2,\ldots,\mu_N)$, the number of elements in $gH$ with cycle-type $\mu$ is given by the expression
\begin{equation} \label{eq:intersection_young_conj}
\frac{|H|\cdot |C_\mu|}{N!}\Phi\left[\prod_{k=1}^{N}\left(\sum_{\{i: d_i|k\}}\epsilon_i^{\frac{k}{d_i}}d_i\right)^{\mu_k}\right]
\end{equation}
where $C_\mu$ is the conjugacy class of all elements with cycle-type $\mu$, and $\epsilon_i$ and $\Phi$ are as above.
\end{thm}
\begin{example}[\textbf{Expected number of $k$-cycles in $gH\subseteq S_N$}]
Let $H$ and $g$ be as in the statement of Theorem \ref{thm:average_young}. Compute the expected number of $k$-cycles in an element $\sigma\in gH$, i.e. the average of $X_k$ over $gH$, using Equation \ref{eq:young_statistics} with $\mu_k=1$ and $\mu_j=0$ for all $j\neq k$. It takes the form
$$
\frac{1}{k}\Phi\left(\sum_{\{i: d_i|k\}}\epsilon_i^{k/d_i}d_i\right) = \frac{1}{k}\sum_{\substack{\{i: d_i|k\}\\ d_i r_i\geq k}} d_i.
$$
\end{example}
\begin{proof}[Proof of Theorem \ref{thm:average_young}]
The proof given for Theorem \ref{thm:intro_evaluating_xmu} applies more generally in this case: replacing the set of roots $\{\alpha_{i,k: 1\leq i\leq n, k\in \Z/d_i\Z}\}$ by a formal set of pairs $\{(i,k): 1\leq i\leq n,k\in \Z/d_i\Z \}$, and the Frobenius permutation by the formal permutation $\tau: (i,k)\mapsto (i, k+1)$, the proof proceeds as presented above.

This shows that our derivation leading up to Equation \ref{eq:final_form_for_young_cosets} applies and the expected values of $\binom{X}{\mu}$ on $gH$ are given by the generating function
$$
\Phi \left[\exp\left( \sum_{k=1}^\infty \sum_{\{i: d_i|k\}}\epsilon_i^{k/d_i}d_i \, \frac{t_k}{k} \right)\right].
$$
Extract the $\mu$-th coefficient by derivation with respect to $\textbf{t}$. The resulting Taylor coefficient is
$$
\Phi \prod_{k=1}^{\infty} \frac{1}{k^{\mu_k}\mu_k!}\left(\sum_{\{i: d_i|k\}}\epsilon_i^{k/d_i}d_i\right)^{\mu_k}.
$$
Now use Fact \ref{fact:expectation} to substitute $\prod_{k=1}^{\infty} \frac{1}{k^{\mu_k}\mu_k!} = \mathbb{E}_{S_N}\left[\binom{X}{\mu}\right]$ and arrive at the final form of our equation.

Lastly, the statement regarding the case $\|\mu\| = N$ follows from the observation that $\binom{X}{\mu}$ is the indicator function of $C_{\mu}$.
\end{proof}


\begin{thebibliography}{15}


\bibitem[ABR]{ABR} J. C. Andrade, L. Bary-Soroker, and Z. Rudnick. {Shifted convolution and the Titchmarsh divisor problem over $F_q[t]$}, \emph{Phil. Trans. R. Soc.} 373, no. 2040 (2015).

\bibitem[CF]{CF} T. Church and B. Farb. {FI-modules and stability for representations of symmetric groups}, \emph{Advances in Math.} 245, no. 1 (2013), pp. 250-314.

\bibitem[CEF1]{CEF-FI} T. Church, J. Ellenberg and B. Farb. {FI-modules and stability for representations of symmetric groups}, \emph{Duke Math. J.} 164, no. 9 (2015), pp. 1833-1910.

\bibitem[CEF2]{CEF-pointcounts} T. Church, J. Ellenberg, and B. Farb. {Representation stability in cohomology and asymptotics for families of varieties over finite fields}, \emph{Contemporary Math.} 620 (2014), pp. 1-54.

\bibitem[SGA4\textonehalf]{SGA} P. Deligne. {SGA 4 1/2-Cohomologie étale}, \emph{Lecture Notes in Math.} 569, Springer-Verlag, New York/Berlin (1977), pp. 52.

\bibitem[Ga1]{Ga-FI} N. Gadish. Categories of FI type: a unified approach to generalizing representation stability and character polynomials,  \emph{J. of Algebra} 480 (2017), pp. 450-486.

\bibitem[Ga2]{Ga-arrangements} N. Gadish. Representation Stability for Families of Linear Subspace Arrangements, \emph{preprint}(2016), {arXiv:1603.08547}.

\bibitem[Gr]{Groth} A. Grothendieck. {Formule de Lefschetz et rationalité des fonctions L}, \emph{Séminaire Bourbaki} 9 (1964), pp. 41-55.

\bibitem[Se1]{Serre} J.P. Serre. Zeta-functions and L-functions, \emph{Uspekhi Matematicheskikh Nauk} 20.6 (1965), pp. 19-26.

\bibitem[Se2]{Se} J.P. Serre. {Linear representations of finite groups}, \emph{Graduate Texts in Math.} 42, Springer-Verlag, New York/Heidelberg (1977), pp. 94.

\end{thebibliography}

\end{document}